\documentclass[12pt]{amsart}

\usepackage{amsmath,amsthm,amsfonts,latexsym,amscd,amssymb,enumerate}

\usepackage{amssymb}

\newcommand{\R}{\mathbb{R}}

\newcommand{\al}{\alpha}
\newcommand{\g}{\gamma}
\newcommand{\eps}{\varepsilon}

\newcommand{\w}{\wedge}

\newcommand{\la}{\langle}
\newcommand{\ra}{\rangle}

\renewcommand{\prod}[2]{\la\,#1 , #2\,\ra}
\newcommand{\ip}[2]{\ensuremath{\langle #1,#2\rangle}}

\newtheorem{prop}{Proposition}

\newtheorem{thm}{Theorem}

\begin{document}

\author{Luis Guijarro}
\address{ Department of Mathematics, Universidad Aut\'onoma de Madrid, and ICMAT CSIC-UAM-UCM-UC3M}
\curraddr{}
\email{luis.guijarro@uam.es}
\thanks{The first author was supported by research grants MTM2008-02676-MCINN and MTM2011-22612}

\author{Gerard Walschap}
\address{Department of Mathematics, University of Oklahoma.}
\curraddr{}
\email{gerard@ou.edu}
\thanks{}
\title{The curvature operator at the soul}


\begin{abstract} We prove two splitting theorems, one topological, the other metric, for open manifolds with nonnegative sectional curvature.
\end{abstract}

\maketitle

	In this note, we study the behavior of the curvature operator $\rho:\Lambda^2(TM)\to\Lambda^2(TM)$ along the soul of an open manifold $M$ with nonnegative sectional curvature $K$. We adopt the convention that the curvature tensor $R$ is given by $\ip{R(x,y)y}x=K(x,y)$ for orthonormal vectors $x$ and $y$. Thus, our curvature tensor agrees with that of \cite{CE}, but differs from \cite{DC} as to sign. Recall that the curvature operator is the self-adjoint endomorphism given on decomposable elements by 
	\[\prod{\rho(x\w y)}{z\w w}=\prod{R(x,y)w}{z}, \qquad x,y,z,w\in M_p,\quad p \in M.
	\] 
Given a positive integer $n$, the curvature operator is said to be \emph{$n$-nonnegative} (resp.~\emph{$n$-positive}) on $U\subset M$ if the sum of any $n$ eigenvalues	 of $\rho$ is nonnegative (resp.~positive) at every point of $U$. 1-nonnegative is synonymous with nonnegative. 

Compact manifolds with 2-positive  and those with 2-nonnegative curvature operator were  classified in \cite{BW} and \cite{NW} respectively. Open (i.e., complete, noncompact) spaces $M$ with nonnegative curvature operator also have nonnegative sectional curvature, and therefore admit a soul $\Sigma$ in the sense of Cheeger and Gromoll. In \cite{No}, it was shown that the universal cover $\tilde M$ splits as an isometric product of $\tilde\Sigma$ with a manifold diffeomorphic to Euclidean space. Here, we generalize this result to spaces with 3-nonnegative curvature operator. Of course, one must also assume nonnegative sectional curvature in this case, since otherwise there is no soul. On the other hand, the argument only requires the condition regarding the curvature operator to hold along the soul.

As yet another indication that the structure of these spaces is determined at the soul, we show that if the scalar curvature of $M$ is small enough and $\Sigma$ is simply connected, then $M$ is diffeomorphic to a product of $\Sigma$ with Euclidean space.

We would like to thank the referee for pointing out a mistake in a preliminary proof of Theorem 1, and Lorenzo Sadun for referring us to \cite{Uh}.

\section{Curvature relations}\label{relations}

We begin by collecting a few facts about the curvature tensor $R$; these are not essentially new, but proofs are included for convenience of the reader.
	First of all, observe that in spaces of nonnegative curvature, compact or not,
	$
	R(x,y)y=0
	$
	whenever $K(x,y)=0$. To see this, consider the function
	 \[
	t\mapsto f(t)=\prod{R(x,y+tu)(y+tu)}{x}.\]
	Expanding this formula for $f$  and using $f'(0)=0$ then yields the result. 
	
In what follows, we restrict ourselves to the curvature at the soul $\Sigma$ of an open manifold with $K\ge0$. Recall from \cite{CG} that any plane spanned by a vector tangent to $\Sigma$ and one orthogonal to 	it has zero curvature. 

	Given $p\in \Sigma$, let $x$, $y\in\Sigma_p$, and $u$, $v\perp\Sigma_p$. Then
		\begin{equation}\label{tag1.1}
	R(x,y)u = 2R(x,u)y, \qquad R(u,v)x=2R(u,x)v.	\end{equation}
	Indeed, $R(x+y,u)(x+y)=0$ since the plane spanned by $x+y$ and $u$ is flat. Expanding this expression and using again $R(x,u)x=R(y,u)y=0$ yields $R(y,u)x=-R(x,u)y$. Substituting in the Bianchi identity $R(x,y)u+R(y,u)x+R(u,x)y=0$ then implies the first assertion. The proof of the second one is similar.
	
	The other relation we will need is the following:
\begin{equation}\label{tag1.3}
	\prod{R(x,y)y}{x}\prod{R(u,v)v}{u}\geq \frac{9}{4}\prod{R(x,y)u}{v}^2.    
	\end{equation}

 For the proof,  let $\al,\beta,\g,\delta\in\R$,
$e=\al x+\beta u$, $f=\g y+\delta v$, and expand $\prod{R(e,f)f}{e}$ to obtain 
\[\begin{split}
\prod{R(e,f)f}{e}&=(\al\g)^2\prod{R(x,y)y}{x}+3\al\beta\g\delta\prod{R(x,y)v}{u}\\
&\quad+(\beta\delta)^2\prod{R(u,v)v}{u}.
\end{split}
\]
The right side may be written as $Q(\mathbf{x})$, with $\mathbf{x}=(\al\g,\beta\delta)\in\R^2$, and $Q$ a quadratic form on the plane. But $Q$ is nonnegative definite, which is exactly \eqref{tag1.3}.

\section{A metric splitting theorem}
The main result of this section will use the following fact from linear algebra.
\begin{prop}
Let $\rho:E\to E$ denote a self-adjoint endomorphism of an $n$-dimensional inner product space $E$. Then $\rho$ is $k$-nonnegative if and only if
\[\sum_{i=1}^k\ip{\rho v_i}{v_i}\ge 0
\]
 for any $k$ orthonormal vectors $v_i\in E$.
\end{prop}
\begin{proof} Let $\lambda_1\le\dots\le\lambda_n$ denote the eigenvalues of $\rho$, and $\{e_i\}$  a corresponding orthonormal basis of eigenvectors. The only if part is clear since
\[\sum_{i=1}^k\lambda_i=\sum_{i=1}^k\ip{\rho e_i}{e_i}\ge 0.
\]
For the converse, let $v_1,\dots,v_k\in E$ be orthonormal, and write 
\[v_i=\al_{i1}e_1+\dots+\al_{in}e_n, \qquad \al_{ij}\in\R.
\]
We then have
\[
\begin{split}
\sum_{i=1}^k\ip{\rho v_i}{v_i} & = \sum_{i=1}^k\left( \sum_{j=1}^n\lambda_j\al_{ij}^2\right) = \sum_{j=1}^{k-1}\left( \lambda_j\sum_{i=1}^k\al_{ij}^2\right) +\sum_{j=k}^{n}\left( \lambda_j\sum_{i=1}^k\al_{ij}^2\right)\\
&\ge \sum_{j=1}^{k-1}\left( \lambda_j\sum_{i=1}^k\al_{ij}^2\right) + \lambda_k\sum_{j=k}^n\sum_{i=1}^k\al_{ij}^2.
\end{split}
\]
Rewriting the last summation as
\[\sum_{j=k}^n\sum_{i=1}^k\al_{ij}^2=\sum_{j=1}^n\sum_{i=1}^k\al_{ij}^2-\sum_{j=1}^{k-1}\sum_{i=1}^k\al_{ij}^2 = k  -\sum_{j=1}^{k-1}\sum_{i=1}^k\al_{ij}^2,
\]
we conclude that
\begin{equation}\label{Eq:1}
\begin{split}
\sum_{i=1}^k\ip{\rho v_i}{v_i} & \ge \sum_{j=1}^{k-1}\left( \lambda_j \sum_{i=1}^k\al_{ij}^2\right) + \lambda_k\left( k  -\sum_{j=1}^{k-1}\sum_{i=1}^k\al_{ij}^2\right)\\
&=k\lambda_k +\sum_{j=1}^{k-1}\left(( \lambda_j -\lambda_k)\sum_{i=1}^k\al_{ij}^2\right) .
\end{split}
\end{equation}
Next, observe that $\sum_{i=1}^k\al_{ij}^2\le1$ for any $j$. Indeed, extending $v_1,\dots,v_k$ to an orthonormal basis $v_1,\dots, v_n$ of $E$, we have
\[1=\ip{e_j}{e_j}=\sum_{i=1}^n\ip{e_j}{v_i}^2\ge \sum_{i=1}^k\ip{e_j}{v_i}^2=\sum_{i=1}^k\al_{ij}^2.
\]
Substituting this inequality in \eqref{Eq:1} yields
\[\sum_{i=1}^k\ip{\rho v_i}{v_i} \ge k\lambda_k+\sum_{j=1}^k(\lambda_j-\lambda_k) = \lambda_1+\dots+\lambda_k
\]
which is nonnegative by assumption. This establishes the claim.
\end{proof}
\begin{thm} Let $M$ be an open manifold with nonnegative sectional curvature. If the curvature operator of $M$ is 3-nonnegative when restricted to a soul $\Sigma$, then $M$ splits locally isometrically over $\Sigma$.
\end{thm}
\begin{proof}
Denote by $n$ and $k$ the dimension and codimension of the soul, respectively. Suppose the normal bundle of $\Sigma$ is not flat, so that there exist $p\in \Sigma$, $x,y\in\Sigma_p$, $u\perp\Sigma_p$ such that $R(x,y)u\ne0$. If $v$ is a unit vector in direction $R(x,y)u$, and $\alpha=|R(x,y)u|>0$, then
\[\ip{R(x,y)u}v=\alpha>0.
\]
If $\rho$ denotes the curvature operator of $M$, then by \eqref{tag1.1},
\[\begin{split}
\ip{\rho(x\wedge u+ y\wedge v)}{x\wedge u + y\wedge v}&=2\ip{\rho(x\wedge u)}{y\wedge v}= 2\ip{R(x,u)v}y\\
&=-2\ip{R(x,u)y}v=-\ip{R(x,y)u}v\\
&=-\alpha,
\end{split}\]
and similarly,
\[
\ip{\rho(x\wedge v- y\wedge u)}{x\wedge v - y\wedge u}=-\alpha.
\]
Furthermore,
\[\ip{\rho(x\wedge v+y\wedge u)}{x\wedge v+y\wedge u}=\alpha.
\]
Notice that the three bivectors above are mutually orthogonal of length $\sqrt{2}$. Normalizing them, we obtain three orthonormal bivectors $\xi_1,\xi_2,\xi_3$ satisfying 
\[\sum_{i=1}^3\ip{\rho(\xi_i)}{\xi_i}=\frac12(-\alpha-\alpha+\alpha)<0,\]
which contradicts the Proposition in the case $k=3$.

Summarizing, if $\rho$ is 3-nonnegative along the soul, then the normal bundle of the soul must be flat. The main result in \cite{St} now implies that $M$ is locally a metric product with $\Sigma$ as one of the factors. 
\end{proof}
It is worth noting that it is essential for the curvature operator condition to hold along the soul. One can construct examples where the curvature operator is nonnegative everywhere outside a compact set, but where $M$ does not split:
consider the Hopf action of $S^1$ on $S^3$, and let $g$ denote a nonnegatively curved rotationally symmetric metric on $\R^2$ that is isometric to $S^1\times[1,\infty)$ outside a compact set. Construct the quotient $M=S^3\times_{S^1}\R^2$. Since the diagonal action of $S^1$ on the metric product $S^3\times(\R^2,g)$ is by isometries, we may endow $M$ with the unique metric for which the projection  $S^3\times(\R^2,g)\to M$ is a Riemannian submersion. $M$ is then a nonnegatively curved manifold with soul $S=S^3\times_{S^1}\{0\}$. Away from a tubular neighbourhood of $S$, $M$ is isometric to 
$(S^3\times_{S^1}S^1)\times [1,\infty)$ and on that set with nonnegative sectional curvature, the curvature operator is also nonnegatively curved: indeed, $S^3\times_{S^1}S^1$ has dimension 3, so that every bivector is decomposable. $M$ is, however, the total space of a nontrivial bundle over $S^2$ (its Euler class is not trivial), so that its curvature operator cannot be nonnegative everywhere. 

\section{A differentiable splitting theorem}

\begin{thm}
Let $M$ be an open nonnegatively curved manifold with simply connected soul $\Sigma$, and choose some $r>(\dim \Sigma)/2$. There exists some $\eps_0=\eps_0(\Sigma,r)$ such that if 
the scalar curvature $s_M$ of $M$ along $\Sigma$ satisfies
$$
\left(\int_{\Sigma}\,s_M^r\,dv\right)^{1/r} < \eps_0 
$$
then $M$ is diffeomorphic to $\Sigma\times\R^k$.
\end{thm}
\begin{proof}
We begin with the determinant inequality (\ref{tag1.3}). For any $p\in \Sigma$, take the trace of this inequality over  an orthonormal basis $\{x_i\}$of $\Sigma_p$ and $\{v_j\}$ of its orthogonal complement to obtain
\begin{equation}
\sum_{i,\dots,l}\ip{R(x_i,x_j)u_k}{u_l}^2 \leq \frac{4}{9}\sum_{i,j}K(x_i,x_j)\cdot\sum_{k.l}K(u_k,u_l).
\end{equation} 
Since planes spanned by one vector tangent, and another vector orthogonal to $\Sigma$ are flat, we obtain 
\begin{equation}
|R^\nabla|^2\leq c\cdot {s_M}^2.
\end{equation}
Here $R^\nabla$ denotes the curvature tensor of the normal bundle of $\Sigma$ (which is just the restriction of the curvature tensor of $M$ since $\Sigma$ is totally geodesic). The term $c$ is due to the constants appearing in the definition of the scalar curvature, and may, without affecting the remainder of the argument, be assumed to equal 1; raising to the $r/2$-th power and integrating over $\Sigma$ yields
\begin{equation}
\|R^\nabla\|_r=\left(\int\,|R^\nabla|^r\right)^{1/r} \leq
\left(\int\,|s_M|^r\right)^{1/r}=\|s_M\|_r.
\end{equation}

Assume now the statement is false; then $M$ is not diffeomorphic to the total space of a flat bundle, but there exists a sequence $\{g_i\}$ of metrics  over it having nonnegative curvature and scalar curvature with soulwise norm 
$\|s^{(i)}_M\|_r <1/i$. Thus, the corresponding curvature tensors on the normal bundle satisfy $\|R^{\nabla_i}\|_r< 1/i$, and a result of Uhlenbeck
\cite{Uh} implies the existence of a sequence of gauge transformations $g_i$ on the normal bundle $E$ of the soul such that 
${g_i}^{-1}\circ\nabla_i\circ g_i$ converges to the flat connection on $E$, contradicting the assumption that the bundle is not flat.\end{proof}

\end{document}